\documentclass{article}
\usepackage[utf8]{inputenc}

\usepackage[american]{babel}
\usepackage[round]{natbib}

\usepackage{mathtools}
\usepackage{amsmath,amsfonts,amssymb,dsfont,bm,amsthm}
\usepackage{csquotes}
\usepackage{xcolor}
\usepackage{float}
\usepackage{subcaption}
\usepackage{graphicx}
\usepackage{txfonts}
\usepackage{dsfont}
\usepackage{hyperref}
\usepackage{thm-restate}
\usepackage{soul}
\usepackage{fix-cm}
\usepackage{comment}
\usepackage{soul}
\def\*#1{\mathcal{#1}}
\theoremstyle{definition}

\newtheorem{theorem}{Theorem}
\newtheorem{lemma}{Lemma}
\newtheorem{corollary}{Corollary}
\newtheorem{proposition}{Proposition}

\newtheorem{definition}{Definition}

\newtheorem{remark}{Remark}
\newtheorem{problem}{Problem}

\newtheorem{notation}{Notation}

\usepackage{slashed}
\newcommand{\ind}{\perp\!\!\!\perp}
\newcommand{\leftcomingarrow}{\leftarrow \! \! \! \! \bullet}
\newcommand{\rightcomingarrow}{\bullet \! \! \! \! \rightarrow}
\newcommand{\notind}{\slashed{\ind}}
\newcommand{\V}{\mathcal{V}}
\newcommand{\E}{\mathcal{E}}
\newcommand{\X}{\mathcal{X}}
\newcommand{\Y}{\mathcal{Y}}
\newcommand{\Z}{\mathcal{Z}}

\newcommand{\Anc}{\text{Anc}}
\newcommand{\Desc}{\text{Desc}}
\newcommand{\G}{{\mathcal{G}}}
\newcommand{\Gm}{{\G^m}}
\newcommand{\Gc}{{\G^C}}
\newcommand{\C}{\mathcal{C}}
\newcommand{\U}{\mathcal{U}}

\newcommand{\Gcu}{{{\G^{m,C}_{\text{u}}}}}

\newcommand{\Gt}{{\G^m_{\min}}}
\newcommand{\Root}{\mathrm{Root}}
\newcommand{\A}{\mathcal{A}}
\newcommand{\B}{\mathcal{B}}
\newcommand{\F}{\mathcal{F}}

\usepackage{tikz}
\usetikzlibrary{babel, arrows, decorations, decorations.pathmorphing, decorations.pathreplacing, decorations.shapes, calc}
\usetikzlibrary{arrows.meta, positioning}
\usetikzlibrary{fit}
\usetikzlibrary{shapes.geometric}

\newcommand{\dashleftselfloop}{
    \!\!\!
    \begin{tikzpicture}[baseline=(X.base)]
        \node (X) at (0,0) {$\phantom{X}$};
        \path (X) edge [<->, loop left, looseness=4, in=155, out=205, dashed] node {} (X);
    \end{tikzpicture}
    \!\!\!\!\!\!\!
}

\newcommand{\leftselfloop}{
    \!\!\!
    \begin{tikzpicture}[baseline=(X.base)]
        \node (X) at (0,0) {$\phantom{X}$};
        \path (X) edge [->, loop left, looseness=4, in=155, out=205] node {} (X);
    \end{tikzpicture} 
    \!\!\!\!\!\!\!
}

\usepackage[vlined]{algorithm2e}
\RestyleAlgo{ruled}
\SetKwComment{Comment}{/* }{ */}

\title{Note on the identification of total effect in Cluster-DAGs with cycles}
\author{Clément Yvernes}
\date{March 2025}

\begin{document}
\maketitle

\begin{abstract}
   In this note, we discuss the identifiability of a total effect in cluster-DAGs, allowing for cycles within the cluster-DAG (while still assuming the associated underlying DAG to be acyclic). This is presented into two key results: first, restricting the cluster-DAG to clusters containing at most four nodes; second, adapting the notion of d-separation. We provide a graphical criterion to address the identifiability problem.
\end{abstract}

\section{Definitions and notations}
We follow the notations of \citet{pearl_causality_2009,anand_causal_2023}.

\paragraph{Notations.} A single variable is denoted by an uppercase letter $X$ and its realized value by a small letter $x$. 
A calligraphic uppercase letter $\X$ denotes a set (or a cluster) of variables.  We denote by $Pa(\X)_{\mathcal{G}}$, $\Anc(\X)_{\mathcal{G}}$, and $\Desc(\X)_{\mathcal{G}}$, the sets of parents, ancestors, and descendants for a given graph ${\mathcal{G}}$, respectively. 

\paragraph{Active vertices and paths.} A vertex $V$ is said to be \emph{active} on a path relative to $\Z$ if 1) $V$ is a collider and $V$ or any of its descendants are in $\Z$ or 2) $V$ is a non-collider and is not in $\Z$. A path $\pi$  is said to be \emph{active} given (or conditioned on) $\Z$ if every vertex on $\pi$ is active relative to $\Z$. Otherwise, $\pi$   is said to be \emph{inactive}. 

\paragraph{Separated sets.} Given a graph ${\mathcal{G}}$, $\X$ and $\Y$ are d-separated by $\Z$ if every path between $\X$ and $\Y$ is inactive given $\Z$. We denote this d-separation by $\X \ind_{\G} \Y \mid \Z$. Otherwise, $\X$ and $\Y$ are d-connected by $\Z$.

\paragraph{Mutilated graph.} The mutilated graph 
 ${\mathcal{G}}_{\overline{\X}\underline{\Z}}$ is the result of
removing from a graph ${\mathcal{G}}$ edges 
with an arrowhead into $\X$ (e.g., $A \rightarrow \X$, $A \leftrightarrow \X$),  
and edges with a tail from $\Z$ (e.g., $A \leftarrow \Z$).

\paragraph{Structural Causal Models and associated graphs.} Formally, a {Structural Causal Model} (SCM) $\mathcal{M}$ is a 4-tuple $\langle \U, \V, \mathcal{F}, P(\U)\rangle$, where $\U$ is a set of exogenous (latent) variables and $\V$ is a set of endogenous (measured) variables. $\mathcal{F}$ is a collection of functions $\{f_i\}_{i=1}^{|\V|}$ such that each endogenous variable $V_i\in\V$ is a function $f_i\in\mathcal{F}$ of $\U_i\cup Pa(V_i)$, % to $V_i$,
where $\U_i\subseteq\U$ and $Pa(V_i)\subseteq\V\setminus V_i$. The uncertainty is encoded through a probability distribution over the exogenous variables, $P(\U)$.
Each SCM $\mathcal{M}$ induces a directed acyclic graph (DAG) with bidirected edges -- or an acyclic directed mixed graph (ADMG) -- $G(\V, \*E = (\E_{D},\E_{B}))$, known as a \emph{causal diagram}, that encodes the structural relations among $\V\cup\U$, where every $V_i \in \V$ is a vertex. We potentially distinguish edges $\mathcal{E}$ if they are directed $\E_D$ or bidirected $\E_B$. There is a directed edge $(V_j \rightarrow V_i)$ for every $V_i \in \V$ and $V_j \in Pa(V_i)$, and there is a dashed bidirected edge $(V_j \dashleftrightarrow V_i)$ for every pair $V_i, V_j \in \V$ such that $\U_i \cap \U_j \neq \emptyset$ ($V_i$ and $V_j$ have a common exogenous parent). 

\paragraph{Interventions and do-operator.} Performing an intervention $\X\!\!=\!\! x$ is represented through the do-operator, \textit{do}($\X\!=\!x$), which represents the operation of fixing a set $\X$ to a constant $x$, and 
induces a submodel $\mathcal{M}_\X$, which is $\mathcal{M}$ with $f_X$ replaced to $x$ for every $X \in \X$. The post-interventional distribution induced by $\mathcal{M}_\X$ is denoted by $P(\V \setminus \X |do(\X))$.\\ %or simply $P_\X(\V)$. 

We work on the Cluster-DAG {framework} introduced in \cite{anand_causal_2023}. 
Variables are grouped into a set of clusters of variables $\{\*{C}_1, \ldots, \*{C}_k\}$ that forms a partition of $\V$ (in this setting, a cluster may be reduced to a single variable).
The underlying causal diagram over the individual variables in $\V$ may not be available, but we assume that we have the causal knowledge about the clusters. 

\begin{definition}[Cluster-DAG]
    A Cluster-DAG (or C-DAG) $\Gc =\left(\V^C, (\E^C_{D},\E^C_{B})  \right)$ is an ADMG  whose nodes represent a (non empty) cluster of variables. \end{definition}

 The cardinal of each cluster is displayed in the upper left corner of each node as represented in Figure \ref{fig:example_cluster_DAG}, and  for $A$ a cluster, we denote the variable inside $A$ as following~: $A = \left\{A_1, \dots, A_{\#A} \right\}$

 The novelty in this work is that cluster-DAGs are not supposed to be acyclic, as illustrated in Figure \ref{fig:example_cluster_DAG}, which means that the clusters are  not an  admissible partition of the nodes.
We keep the name Cluster-DAG to emphasize that the graph over the variables is acyclic. 

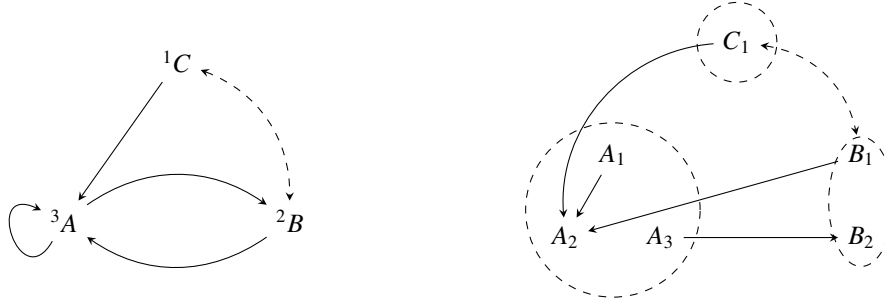
\begin{figure}[t]
    \begin{tikzpicture}[ ->, >=stealth, scale=3]
        % Nodes
        \node (A) at (0, 0) {${}^3A$};
        \node (B) at (1, 0) {${}^2B$};
        \node (C) at (0.5, 0.7071) {${}^1C$};

        % Edges
        \draw (C) -- (A);
        \path (C) edge [<->, dashed, bend left = 35] node {} (B);
        \path (A) edge[bend left=35] (B);
        \path (B) edge[bend left=35] (A);
        \path (A) edge [->, loop left, looseness=5, in=155, out=240] node {} (A);

        \end{tikzpicture} \hfill
            \begin{tikzpicture}[->, >=stealth, scale=1.1]
        % Micro-level Nodes
        \node[anchor=center] (A1) at (0.0000, 0.6667) {$A_1$};
        \node[anchor=center] (A2) at (-.5774, -.3333) {$A_2$};
        \node[anchor=center] (A3) at (0.5774, -.3333) {$A_3$};
        
        \node[anchor=center] (B1) at (3.0000, 0.6667) {$B_1$};
        \node[anchor=center] (B2) at (3.0000, -.3333) {$B_2$};
        
        \node[anchor=center] (C)  at (1.50000, 2.0284) {$C_1$};
        
        \node[draw, dashed, ellipse, scale= 7.0, xscale=1, yscale=1, anchor=center]  at (0, 0) {};
        \node[draw, dashed, ellipse, scale= 1.3, xscale=2, yscale=4, anchor=center]  at (3, 0.1) {};
        \node[draw, dashed, ellipse, scale= 3.2, xscale=1, yscale=1, anchor=center]  at (1.5000, 2.0284) {};

        \path (C) edge[bend right=45] (A2);
        \path (C) edge [<->, dashed, bend left = 35] node {} (B1);
        \draw (A3) -- (B2);
        \draw (A1) -- (A2);
        \draw (B1) -- (A2);

    \end{tikzpicture}
        \caption{Example of Cluster-DAG (left) and a micro-DAG compatible (right).}
        \label{fig:example_cluster_DAG}
\end{figure}

\begin{definition}[Micro Graph Compatible]
    \label{def:compatible_graphs}
Let $\mathcal{V}^c = \{C_1,\ldots, C_k\}$ be a partition of the variables $\mathcal{V}^m = (V_i)_i$. 
    Let $\Gc$ be a cluster-DAG on $\V^C$. An ADMG $\Gm =\left(\V^m, (\E^m_{D},\E^m_{B})  \right)$ on $\V^m$ is a micro graph compatible with $\Gc$, if it satisfies the following properties:
    \begin{itemize}
        \item $\V^m = \bigcup_{V \in \V^C} V = \bigcup_{V \in \V^C}  \left\{V_1, \dots, V_{\#V} \right\}$.
        
        \item For all distinct $A^C, B^C \in \V^C$ , $A^C \rightarrow B^C \in \E^C_{D}$ if and only if there exist $(i,j) \in \left\{1, \dots , \#A^C\right\} \times \left\{1, \dots , \#B^C\right\}$ such that $A_i \rightarrow B_j \in \E^m_{D}$.
        
        \item For all distinct $A^C, B^C \in \V^C$ , $A^C \dashleftrightarrow B^C \in \E^C_{B}$ if and only if there exist $(i,j) \in \left\{1, \dots , \#A^C\right\} \times \left\{1, \dots , \#B^C\right\}$ such that $A_i \dashleftrightarrow B_j \in \E^m_{B}$.
        
        \item For all $A^C \in \V^C$ , $\leftselfloop A^C \in \E^C_{D}$ if and only if there exist distinct $i,j \in \left\{1, \dots , \#A^C\right\}$ such that $A_i \rightarrow A_j \in \E^m_{D}$.
        
        \item For all $A^C \in \V^C$ , $\dashleftselfloop A^C \in \E^C_{B}$ if and only if there exist distinct $i,j \in \left\{1, \dots , \#A^C\right\}$ such that $A_i \dashleftrightarrow A_j \in \E^m_{D}$.

    \end{itemize}
    
    \noindent The set of all micro DAGs compatible with $\Gc$ is written as $\C\left(\Gc\right)$.
\end{definition}

\begin{notation}
\label{notation:base}
    Let $\Gc$ be a cluster-DAG, and let $V$ be a cluster in $\Gc$. We use the following notations:
    \begin{itemize}
        \item When $V$ is seen as a node of $\Gc$, $V$ will be written as $V^C$.
        \item When $V$ is seen as a set of variables of compatible micro DAG $\Gm$, $V$ will be written as $V^m = \{V_1, \dots, V_{\#V} \}$ where the indices follow the topological ordering induced by $\Gm$.
    \end{itemize}  
    We will use the same notations for any intersection or union of cluster.
\end{notation}

\section{Main problem}
We are interested in extending classical causal inference based on do-calculus (see \citep{pearl_causality_2009}) to C-DAG.
In \cite{anand_causal_2023}, they provide the foundations and machinery for valid probabilistic and causal inferences for cluster-DAGs, akin to Pearl’s d-separation and classical do-calculus for when such a coarser graphical representation of the system is provided based on the limited prior knowledge available.
Our goal is to derive the same analysis, for rung 2 of Pearl's causal ladder  when C-DAGs are allowed to be cyclic in whole generality. 

As discussed in \cite{anand_causal_2023}, a naive approach to causal inference with cluster variables—e.g., identifying $Q = P(\*C_i|do(\*C_k))$—proceeds as follows: first, enumerate all causal diagrams compatible with $\Gc$; then, assess the identifiability of $Q$ in each diagram; finally, return  $P(\*C_i|do(\*C_k))$ if all diagrams yield the same expression, and “non-identifiable” otherwise. However, in practice, this approach becomes intractable in high-dimensional settings.

\noindent{Here we shall answer in a surprising simple way to the two following fundamental questions : } 
Q1 : Can valid inferences be performed about cluster variables using C-DAGs directly, without going through exhaustive enumeration? Q2 : What properties of C-DAGs are shared by all the compatible causal diagrams? 
We formalize this in Problem \ref{pb:main_goal}.
%The ultimate goal of this work is to find out an efficient way to tell if a specific do-calculus rule applies in all compatible graphs without enumerating all of them. To do so we aim at finding an efficient way to tell if d-separation holds in all compatible graphs under some mutilation. Problem \ref{pb:main_goal} formalizes this idea.

\begin{problem}
\label{pb:main_goal}
    Let $\Gc$ be an admissible cluster-DAG.  We aim to find out an efficient way to characterise over $\Gc$ the following property: let $\X^C,\Y^C$ and $\Z^C$ be pairwise distinct subsets of nodes of $\Gc$ (subset of clusters).  Let $\mathcal{A}$ and $\mathcal{B}$ be subsets of nodes of $\Gc$.
    \begin{equation*}
        \X^C \ind_{\Gc_{\overline{\mathcal{A}^C} \underline{\mathcal{B}^C}}} \Y^C \mid \Z^C \quad \overset{\text{\tiny def}}{\equiv} \quad
        \forall \Gm \in \C\left(\Gc\right) ~~
        \X^m \ind_{\Gm_{\overline{A^m} \underline{B^m}}} \Y^m \mid \Z^m.
    \end{equation*}
\end{problem}
\section{A useful characterisation of d-separation}
{To tackle Problem~\ref{pb:main_goal}, we first prove a characterization of the d-separation  for ADMGs (not necessarily related to cluster-DAGs). It will be needed in the sequel, since in our general setting the graphs associated to our cluster-DAGs are ADMGs.}

\begin{definition}
    \label{def:structure_of_interest}
    A \emph{structure of interest} $\sigma$ is an ADMG, with a single connected component, in which each node $V$ satisfies the following property:
    \begin{itemize}
        \item $V$ has at most one outgoing arrow, or,
        \item $V$ has two outgoing arrows but no oncoming arrow.
    \end{itemize}
\end{definition}

We remark that a path is a structure of interest.
We draw in Figure \ref{fig:structure_of_interest} an example of a structure of interest. 

\begin{definition}
\label{def:root}
    Let $\G$ be an ADMG. The roots of $\G$, denoted as $\Root(\G)$, is the set of vertices that has no child.
\end{definition}

\begin{definition}
\label{def:connecting_structure_of_interest}
    Let $\G=(\V,\E)$ be a mixed graph. Let $\X, \Y, \Z$ be pairwise disjoint subsets of $\V$. We say that a structure of interest $\sigma$ connects $\X$ and $\Y$ under $\Z$ if the following conditions hold:
    \begin{itemize}
        \item $\sigma \subseteq \G$ \hfill \textit{(subgraph)}
        \item $\X \cap \sigma \neq \emptyset$ and $\Y \cap \sigma \neq \emptyset$ \hfill \textit{(connects $\X$ and $\Y$)}
        \item $\Root(\sigma) \subseteq \Z \cup \X \cup \Y$ and  \hfill \textit{(all vertices are ancestors of $\Z \cup \X \cup \Y$)}
        \item $(\sigma \setminus \Root(\sigma)) \cap \Z = \emptyset$ \hfill \textit{(neither chains nor forks are in $\Z$)}
    \end{itemize}
\end{definition}

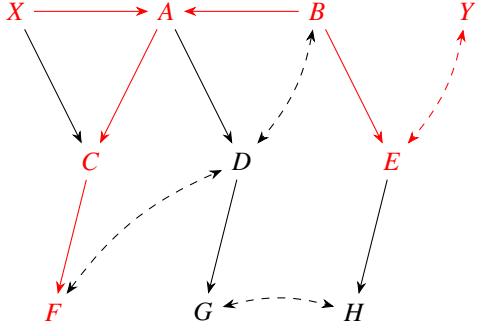
\begin{figure}[ht]
    \centering
    \begin{tikzpicture}[->, >=Stealth]

    \node[text=red] (X) at (0,0) {$X$};
    \node[text=red] (A) at (2,0) {$A$};
    \node (B)[text=red] at (4,0) {$B$};
    \node (Y)[text=red] at (6,0) {$Y$};
    \node[text=red] (C) at (1,-2) {$C$};
    \node (D) at (3,-2) {$D$};
    \node (E)[text=red] at (5,-2) {$E$};
    \node[text=red] (F) at (0.5,-4) {$F$};
    \node (G) at (2.5,-4) {$G$};
    \node (H) at (4.5,-4) {$H$};

    \draw[->, red] (X) -- (A);
    \draw[->, red] (B) -- (A);
    \draw[->, red] (A) -- (C);
    \draw[->] (X) -- (C);
    \draw[->] (A) -- (D);
    \draw[->, red] (B) -- (E);
    \draw[->, red] (C) -- (F);
    \draw[->] (D) -- (G);
    \draw[->] (E) -- (H);
    
    \draw[<->, dashed] (B) to[bend left=15] (D);
    \draw[<->, red, dashed] (Y) to[bend left=15] (E);
    \draw[<->, dashed] (F) to[bend left=15] (D);
    \draw[<->, dashed] (G) to[bend left=15] (H);
    
    \end{tikzpicture}
    \caption{An ADMG in which there is a structure of interest in \textcolor{red}{red} that connects $X$ and $Y$ under $\{F, E\}$.}
    \label{fig:structure_of_interest}
\end{figure}

\begin{theorem}[d-connection with structures of interests]
\label{th:new_d_sep}
    Let $\G$ be an ADMG. Let $\X, \Y, \Z$ be pairwise disjoint subsets of nodes of $\G$. The following properties are equivalent:
    \begin{enumerate}
        \item $\X \notind_{\G} \Y \mid \Z$. \label{th:new_d_sep:1}
        \item $\G$ contains a structure of interest $\sigma$ that connects $\X$ and $\Y$ under $\Z$. \label{th:new_d_sep:2}
    \end{enumerate}
\end{theorem}

\begin{proof}
Let us prove the two implications:
\begin{itemize}
    \item $\ref{th:new_d_sep:1} \Rightarrow \ref{th:new_d_sep:2}$: If \(\X \notind_{\G} \Y \mid \Z\), then, by definition, the ADMG \(\G\) contains a path \(\pi\) connecting \(\X\) and \(\Y\). For each collider \(C\) on \(\pi\), there exists a directed path \(\pi_C\) from \(C\) to a vertex in \(\Z\). Without loss of generality, we may assume that each such path \(\pi_C\) meets \(\Z\) only at its terminal vertex. Define the subgraph
    \[
    \F \coloneqq \pi \cup \bigcup_{C\, \text{collider on } \pi} \pi_C.
    \]
    By construction, \(\F\) has a single connected component and satisfies the following properties:
    \begin{itemize}
        \item \(\F \subseteq \G\).
        \item \(\X \cap \F \neq \emptyset\) and \(\Y \cap \F \neq \emptyset\).
        \item \(\Root(\F) \subseteq \Z \cup \X \cup \Y\).
        \item \((\F \setminus \Root(\F)) \cap \Z = \emptyset\).
    \end{itemize}
    
    However, some vertices in \(\F\) may violate the conditions of Definition~\ref{def:structure_of_interest}, preventing \(\F\) from being a structure of interest. Let \(V\) be such a vertex. Since \(\pi\) uses at most two arrows around \(V\), we remove from \(\F\) any edges adjacent to \(V\) that are not used by \(\pi\). We then consider \(\F' \subseteq \F\), the connected component of \(\F\) containing both \(X\) and \(Y\).\footnote{This procedure preserves the path \(\pi\), so \(X\) and \(Y\) remain connected.} If \(V\) remains in \(\F'\), then it now satisfies the structural conditions of Definition~\ref{def:structure_of_interest}. Moreover, \(\F'\) inherits the following properties:
    \begin{itemize}
        \item \(\F' \subseteq \F \subseteq \G\),
        \item \(\F' \cap \X \neq \emptyset\) and \(\F' \cap \Y \neq \emptyset\),
        \item \(\Root(\F') \subseteq \Root(\F) \subseteq \Z \cup \X \cup \Y\),
        \item \((\F' \setminus \Root(\F')) \subseteq (\F \setminus \Root(\F))\), hence \((\F' \setminus \Root(\F')) \cap \Z = \emptyset\).
    \end{itemize}
    
    We repeat this procedure for all remaining vertices that violate the definition. Since each step strictly reduces the number of such violations, the process terminates in finitely many steps. We thus obtain a subgraph \(\sigma \subseteq \G\) that satisfies all conditions of Definition~\ref{def:structure_of_interest} and connects \(\X\) and \(\Y\) under \(\Z\).

    \item $\ref{th:new_d_sep:2} \Rightarrow \ref{th:new_d_sep:1}$: By definition, $\sigma$ has a single connected component. Thus $\X$ and $\Y$ are connected by $\sigma$. Let $\pi$ be a path from $\X$ and $\Y$ in $\sigma$. Let us first prove that without loss of generality, we can assume that all colliders on $\pi$ are ancestors of $\Z$.
    If it's not the case, since $\Root(\sigma) \subseteq \Z \cup \X \cup \Y$, all colliders that are not ancestors of $\Z$ are ancestors of $\X \cup \Y$. Without loss of generality, assume that a collider \(C\) on \(\pi\) is an ancestor of \(\X\). Thus, $\sigma$ contains a directed path $\pi^1$ from $C$ to $\X$. Let $\pi^2$ be the subpath of $\pi$ between $C$ and $\Y$. Since \(C\) belongs to both \(\pi^1\) and \(\pi^2\), these two paths intersect at \(C\), and possibly at other vertices. Let $T$ be the last vertex of $\pi^1$ that is in $\pi^1 \cap \pi^2$. Let us consider $\pi_1'$, the subpath of $\pi_1$ from $T$ to $\X$ and $\pi_2'$ the subpath of $\pi^2$ from $T$ to $\Y$. Let $\pi' \coloneqq \pi_1' \cup \pi_2'$. $\pi'$ is a path. Moreover, between $X$ and $T$,  $\pi'$ is a directed path, and, after $T$ to $\Y$, $\pi'$ is a subpath of $\pi_2$. Therefore, $\pi'$ contains at least one fewer collider than $\pi$ that is not an ancestor of $\Z$. Repeating this procedure iteratively allows us to construct a path in \(\sigma\) from \(\X\) to \(\Y\) in which all colliders are ancestors of \(\Z\).

    Finally, note that all forks and chains on \(\pi\) are not roots by definition, and hence do not belong to \(\Z\). Therefore, the resulting path \(\pi\) is \(d\)-connecting between \(\X\) and \(\Y\) given \(\Z\).
\end{itemize}

\end{proof}

\newpage
\section{Basic Properties}
In this section, we outline some fundamental properties of C-DAGs that provide deeper insight into their structure and behavior.
%\textcolor{red}{Proof are done rapidly.}

\begin{proposition}
    Let $\Gc$ be a cluster-DAG. $\C\left(\Gc\right) \neq \emptyset$ if and only if $\Gc$ does not contain any cycle on cluster of size $1$. When this is true, we say that $\Gc$ is admissible.
\end{proposition}

\begin{proof}
    Let us prove the two implications:
    \begin{itemize}
        \item $\Rightarrow$: If $\Gc$ contains the cycle ${}^1A^1 \rightarrow \cdots \rightarrow {}^1A^1$, then any compatible graph would contain the cycle $A_1 \rightarrow \cdots \rightarrow A_1$, which is not allowed by Definition \ref{def:compatible_graphs}.
        \item $\Leftarrow$: If $\Gc$ does not contain any cycle on cluster of size $1$. Let us construct a compatible graph. For all cluster $V$ of size $1$, we put all incoming and outgoing edges at $V_1$. This does not create a cycle because, otherwise, $\Gc$ would contain a cycle on cluster of size $1$. For all other clusters $V$, as they are at least of size $2$, we can deal with $V_1$ and $V_2$. We put all outgoing edges at $V_1$ and all incoming edges at $V_2$. This does not create a cycle because there is no chain in $V^m$.
    \end{itemize}
\end{proof}

\begin{proposition}
\label{prop:remove_arrow}
    Let $\Gc$ be an admissible cluster-DAG, $\Gm$ be a compatible graph with $\Gc$ and $V^C$ and $W^C$ be nodes of $\Gc$, ie clusters. If $\Gm$ contains two similar (same type) arrows between $V^m$ and $W^m$, then removing one of these arrows create another compatible graph.
\end{proposition}

\begin{proof}
    By definition, only one arrow between $V^m$ and $W^m$ is necessary.
\end{proof}

\begin{proposition}[Bidirected Arrows are not a big deal]
\label{prop:bidirected_easy}
    Let $\Gc$ be an admissible cluster-DAG and $A$ and $B$ two clusters. The following properties hold:
    \begin{itemize}
        \item For any $\Gm \in \C\left(\Gc\right)$, the micro graph obtained from $\Gm$ by removing all bidirected edges belongs to $\C(\Gc')$, where $\Gc'$ denotes the cluster-DAG obtained from $\Gc$ by removing all bidirected edges.
        \item For any $\Gm \in \C\left(\Gc\right)$, adding the bidirected edge $A_i \dashleftrightarrow B_j$, where $A_i \neq B_j$, results in a micro graph that belongs to $\C(\Gc')$, where $\Gc'$ is the cluster-DAG obtained from $\Gc$ by adding the bidirected edge $A \dashleftrightarrow B$.
    \end{itemize}
\end{proposition}

\begin{proof}
    It follows directly from Definition \ref{def:compatible_graphs}.
\end{proof}

\begin{proposition}[Self-loops are not a big deal]
    Let $\Gc$ be an admissible cluster-DAG. The following properties hold:
    \begin{itemize}
        \item For any $\Gm \in \C\left(\Gc\right)$, the micro graph obtained from $\Gm$ by removing all intra-cluster edges belongs to $\C(\Gc')$, where $\Gc'$ denotes the cluster-DAG obtained from $\Gc$ by removing all self loops.
        
        \item For any $\Gm \in \C\left(\Gc\right)$ and any cluster ${}^{\geq 2}V$, there exists distinct $V_i, V_j \in V$ such that  adding the arrow $V_i \rightarrow V_j$ results in a micro graph that belongs to $\C(\Gc')$, where $\Gc'$ is the cluster-DAG obtained from $\Gc$ by adding the self loop $\leftselfloop V$.
    \end{itemize}
\end{proposition}

\begin{proof}
    The first property holds directly from Definition \ref{def:compatible_graphs}. For the second one, we know that $V_1 \in \Anc(V_2, \Gm)$ and $V_2 \in \Anc(V_1, \Gm)$ cannot hold simultaneously because otherwise, $\Gm$ would contain a cycle. Thus there is always a correct choice between $V_1 \rightarrow V_2$ and $V_2 \rightarrow V_1$.
\end{proof}

\section{d-separation in C-DAGs}

%We do not solve Problem \ref{pb:main_goal} directly. 
We start by solving d-separation without mutilation. %Problem \ref{pb:d_sep} formalizes this idea.
%
%\begin{problem}
%\label{pb:d_sep}
    Let $\Gc$ be an admissible cluster-DAG. Let $\X^C,\Y^C$ and $\Z^C$ be pairwise distinct subsets of nodes of $\Gc$. We aim to find out an efficient way to characterise the following property:
    \begin{equation*}
        \X^C \ind_{\Gc} \Y^C \mid \Z^C \quad \overset{\text{\tiny def}}{\equiv} \quad
        \forall \Gm \in \C\left(\Gc\right) ~~
        \X^m \ind_{\Gm} \Y^m \mid \Z^m 
    \end{equation*}
%\end{problem}

%To solve Problem \ref{pb:d_sep}, we work on its negation. Lemma \ref{lemma:reformulation} formalizes this idea.
Thus, by Theorem \ref{th:new_d_sep}, its negation is: 
%\begin{lemma}
%\label{lemma:reformulation}
%    Let $\Gc$ be an admissible cluster-DAG. Let $\X^C,\Y^C$ and $\Z^C$ be pairwise distinct subsets of nodes of $\Gc$. We have:
    \begin{align}
           & \mathcal{X}^C \notind_{\Gc} \mathcal{Y}^C \mid \mathcal{Z}^C \quad \overset{\text{\tiny def}}{\equiv} \quad 
                \exists~  \Gm \in \C\left(\Gc\right) ~~
                    \X^m \notind_{\Gm} \Y^m \mid \Z^m.\nonumber\\ 
 %   \end{equation*}
  %  
   % \noindent is equivalent to:
    %\begin{equation}
      \Leftrightarrow & \exists ~ \Gm \in \C\left(\Gc\right),
            \exists~ \sigma^m \text{ a structure of interest that connects $\X^m$ and $\Y^m$ under $\Z^m$ }.
    \label{eq:to_be_characterised}
    \end{align}
%\end{lemma}
%\begin{proof}
%    It directly follows from the definitions.
%\end{proof}

\subsection{Reducing the size of the cluster-DAG.}

Clusters in Cluster-DAG can be arbitrary large. This could create an issue if we want to work directly in compatible graphs. To solve this issue, Theorem \ref{th:infinity_leq_four} shows that we only need to consider clusters with a maximum size of four. To show Theorem \ref{th:infinity_leq_four}, we need  Propositions \ref{prop:move_arrow_up} and \ref{prop:move_arrow_down}, which show how to construct a compatible graph from another compatible graph. In this section, following the notations of \cite{perkovic_complete_2018}, an arrow ($\leftcomingarrow$) represents either a directed arrow ($\leftarrow$) or a dashed-bidirected arrows ($\dashleftrightarrow$).

\begin{proposition}
\label{prop:move_arrow_up}
    Let $\Gc$ be an admissible cluster-DAG, $\Gm$ be a compatible graph with $\Gc$ and $V^C$ be a node in $\Gc$, ie a cluster. If there exists $i>j$ such that $\Gm$ contains the arrow $V_i \rightarrow W_w$, then there exists a compatible graph $\Gm'$ that contains the arrow $V_j \rightarrow W_w$ and not $V_i \rightarrow W_w$.
\end{proposition} 

\begin{proof}
    %Let $V_i \rightarrow W_w$ be the arrow. 
    If $i > j$, thus by the convention of Notation \ref{notation:base}, we know that $V_j$ is before $V_i$ in the topological order of $\Gm$. Thus, $V_j$ is not a descendant of $\V_i$ in $\Gm$. Thus, $V_j$ is not a descendant of $W_w$ in $\Gm$. Therefore, adding the arrow $V_j \rightarrow W_w$ into $\Gm$ does not create a cycle, because otherwise $V_j$ would be a descendant of $W_w$ in $\Gm$.
\end{proof}

\begin{corollary}
\label{cor:move_fork_up}
    Let $\Gc$ be an admissible cluster-DAG, $\Gm$ be a compatible graph and $V^C$ be a cluster. If $\Gm$ contains a path which contains a fork on $V_i$ with $i>j$, then there exists a compatible graph $\Gm'$ which contains the same path except that the fork is on $V_j$.
\end{corollary}

\begin{proof}
    Apply Proposition \ref{prop:move_arrow_up} twice. 
\end{proof}

\begin{proposition}
\label{prop:move_arrow_down}
    Let $\Gc$ be an admissible cluster-DAG, $\Gm$ be a compatible graph and $V^C$ be a cluster. If there exists $i < j$ such that $\Gm$ contains the arrow $V_i \leftcomingarrow W_w$, where  $\leftcomingarrow$, then there exists a compatible graph $\Gm'$ that contains the arrow $V_{j} \leftcomingarrow W_w$ instead of $V_i \leftcomingarrow W_w$.
\end{proposition} 

\begin{proof}
    %Let $V_i \leftcomingarrow W_w$ be the arrow. 
    $i < j$, thus by the convention of Notation \ref{notation:base}, we know that $W_w$ is not a descendant of $V_j$. Therefore, adding the arrow $V_j \leftcomingarrow W_w$ does not create a cycle.
\end{proof}

\begin{proposition}[From \cite{ferreira2024identifyingmacroconditionalindependencies}]
\label{prop:active_walk}
    Let $\G$ be a mixed graph and $\Z$ be a subset of nodes. Let $\tilde{\pi} = \langle V^1, \dots, V^n \rangle$ be a walk from $X$ to $Y$ in $\G$. If $\tilde{\pi}$ is $\Z$-active, then $\pi = \langle U^1, \dots, U^m \rangle$, where $U^1 = V^1$ and $U^{k+1} = V^{\max \{ i \mid V^i = U^k \}}$, is a  is $\Z$-active path from $X$ to $Y$.
\end{proposition}

\begin{proposition}
\label{prop:reduce_complexity_path}
    Let $\Gc$ be an admissible cluster-DAG, $\Gm$ be a compatible graph and $V^C$ be a cluster. Let $X^C$, $Y^C$ and $Z^C$ be pairwise disjoint subsets of nodes. Let $\pi^m$ be a $\Z^m$-active path from $\X^m$ to $\Y^m$ in $\Gm$. Then there exists a compatible graph $\Gm'$ that contains ${\pi^m}'$, a $\Z^m$-active path from $\X^m$ to $\Y^m$ such that ${\pi^m}'$ neither contains two collider nodes nor two fork nodes nor two chain nodes in $\Gm'$ in the same cluster.
\end{proposition}

\begin{proof}
    First, we show a strategy to reduce the number of similar triplet on $V^m$. We distinguish three cases:
    \begin{enumerate}
        \item \label{proof:prop:reduce_complexity_path:1} If $\pi^m$ contains two colliders $\rightcomingarrow V_i \leftcomingarrow$ and $\rightcomingarrow V_j \leftcomingarrow$ with $i < j$. By applying Proposition \ref{prop:move_arrow_down} twice, there exists a compatible graph $\Gm'$ in which all the aforementioned arrows are pointing toward $V_j$. Thus, $\Gm'$  contains a $\Z^m$-active walk from $\X^m$ to $\Y^m$ which does not contain $V_i$. By Proposition \ref{prop:active_walk}, $\Gm'$ contains a $\Z^m$-active path from $\X^m$ to $\Y^m$ that does not use $\rightcomingarrow V_i \leftcomingarrow$ anymore. Therefore, this strategy reduces the number of collider on $V^m$ by one as long as there is at least two colliders on $V^m$.
        
        \item \label{proof:prop:reduce_complexity_path:2}  If $\pi^m$ contains two forks $\leftarrow V_i \rightarrow$ and $\leftarrow V_j \rightarrow$ with $i < j$. By applying Corollary \ref{cor:move_fork_up}, there exists a compatible graph $\Gm'$ in which all the aforementioned arrows are pointing away from $V_i$. Thus $\Gm'$ contains a $\Z^m$-active walk from $\X^m$ to $\Y^m$. By Proposition \ref{prop:active_walk}, $\Gm'$ contains a $\Z^m$-active path from $\X^m$ to $\Y^m$ that does not use $\leftarrow V_j \rightarrow$ anymore. Therefore, this strategy reduces the number of fork on $V^m$ by one as long as there is at least two forks on $V^m$.
        
        \item \label{proof:prop:reduce_complexity_path:3} If $\pi^m$ contains two chains $\rightarrow V_i \rightarrow$ and $\rightarrow V_j \rightarrow$ with $i < j$. Without loss of generality, we can assume that the arrows around $V_i$ are pointing in the same direction as the order of vertices in the path. Indeed, d-separation between $\X^m$ and $\Y^m$ is symmetric, thus, if it's not the case, we can consider the path from $\Y^m$ to $\X^m$ instead. We distinguish two cases:
        \begin{itemize}
            \item If $V_i$ is before $V_j$ in $\pi^m$. Then, by Proposition \ref{prop:move_arrow_down}, there exists a compatible graph $\Gm'$ in which the arrow pointing to $V_i$ is now pointing toward $V_j$. We distinguish two cases:
            \begin{itemize}
                \item If the arrows around $V_j$ are pointing in the same direction as the order of vertices in the path, then we create a new chain on $V_j$. This creates a path that bypasses $V_i \rightarrow \dots \rightarrow V_j$. This path is a $\Z^m$-active path from $\X^m$ to $\Y^m$ and it does not pass through $V_i$ anymore. In this case, we reduce the number of chains by one.
                
                \item If the arrows around $V_j$ are pointing in the opposite direction of the order of vertices in the path, then we create a collider on $V_j$. The chains on $V_i$ and $V_j$ are not pointing in the same direction. Thus, $\pi^m$ contains a collider between $V_i$ and $V_j$. Let $C_c$ be the first collider of $\pi^m$ before $V^j$. We know that $V_j$ is an ancestor of $C_c$ and that $C_c$ is an ancestor of an element of $\Z^m$. Therefore, the new-collider is not blocking and we have created a $\Z^m$-active path from $\X^m$ to $\Y^m$ that does not pass through $V_i$ anymore. In this case, we reduce the number of chains by two but increase the number of colliders by one.
            \end{itemize}
            
            \item Otherwise, $V_j$ is before $V_i$ in $\pi^m$. We distinguish two cases:
            \begin{itemize}
                \item If the arrows around $V_j$ are pointing in the opposite direction of the order of vertices in the path. By Proposition \ref{prop:move_arrow_up}, there exists a compatible graph $\Gm'$ in which the arrow pointing away from $V_j$ is now pointing away from $V_i$. This creates a new fork on $V_i$ and a new path that bypasses $V_j \leftarrow \dots \rightarrow V_i$. $\pi^m$ is active, thus $V_i \notin \Z^m$. Thus the new fork is not blocking and the new path is indeed a $\Z^m$-active path from $\X^m$ to $\Y^m$. Moreover, this path does not pass through $V_j$ anymore. In this case, we reduce the number of chains by two but increase the number of forks by one.
                
                \item If the arrows around $V_j$ are pointing in the same direction as the order of vertices in the path. Let $F_f$ be the predecessor of $V_i$ in $\pi^m$. We know that $F_f \neq V_j$ because $i < j$. By Proposition \ref{prop:move_arrow_down} there exists a compatible graph $\Gm'$ that contains the arrow $F_f \rightarrow V_j$. Thus $\Gm'$ contains a path that bypasses  $V_j \rightarrow \dots F_f \rightarrow V_j$ by using $\rightarrow V_j \leftarrow F_f \rightarrow V_i$. Since $i < j$, we know that $V_j$ is not an ancestor of $V_i$ in $\Gm$. Thus, $\pi^m$ contains a collider between $V_j$ and $V_j$. Thus, $V_j \in \Anc(\Z^m, \Gm)$. Thus, $\rightarrow V_j \leftarrow$ is not blocking. Moreover $F_f$ is not blocking because otherwise, it would have been blocking $\pi^m$ in $\Gm$. Therefore, the new path is a $\Z^m$-active path from $\X^m$ to $\Y^m$ and it does not use $\rightarrow V_j \rightarrow$ anymore.  In this case, we reduce the number of chains by one and increase the number of colliders by one.
            \end{itemize}
        \end{itemize}
        Therefore, in all the cases, this strategy always strictly reduce the number of chains in $V^m$, but it may increase the number of other triplets.
    \end{enumerate}
    
    To prove the Proposition, we start by applying Strategy \ref{proof:prop:reduce_complexity_path:3} as long as it is possible. Hence, we get a compatible graph that contains a $\Z^m$-active path from $\X^m$ to $\Y^m$ that  does not contain two chains on $V^m$. Then, we use Strategy \ref{proof:prop:reduce_complexity_path:1} and Strategy \ref{proof:prop:reduce_complexity_path:2} as long as it is possible to get a compatible graph $\Gm'$ that contains ${\pi^m}'$, a $\Z^m$-active path from $\X^m$ to $\Y^m$ such that ${\pi^m}'$ does not contain two colliders nor two fork nor two chains on $V^m$.
\end{proof}

\begin{theorem}[Infinity is at most four]
\label{th:infinity_leq_four}
    Let $\Gc$ be an admissible cluster-DAG and $\G^c_{\leq4}$ be the corresponding cluster-DAG where all clusters of size greater than 4 are reduced to size 4. Let $\X^C$, $\Y^C$ and $\Z^C$ be pairwise disjoint subsets of nodes. The following propositions are equivalent:
    \begin{enumerate}
        \item \label{th:infinity_leq_four:1}
        $\exists ~ \Gm \in \C\left(\Gc\right), ~~
            \exists~ \pi^m \text{ path from $\X^m$ to $\Y^m$ s.t. } \Z^m \text{ d-connects } \pi^m.$
        
        \item \label{th:infinity_leq_four:2}
        $\exists ~ \G^{m}_{\leq4} \in \C\left(\G^c_{\leq4}\right), ~~
            \exists~ \pi^m \text{ path from $\X^m$ to $\Y^m$ s.t. } \Z^m \text{ d-connects } \pi^m.$ 
    \end{enumerate}
\end{theorem}

\begin{proof}
    We prove the two implications:
    \begin{itemize}
        \item $\ref{th:infinity_leq_four:2} \Rightarrow \ref{th:infinity_leq_four:1}$: We can add as many vertices without arrows as necessary.
        
        \item $\ref{th:infinity_leq_four:1} \Rightarrow \ref{th:infinity_leq_four:2}$:Let \(\Gm\) be a compatible graph containing a \(\Z^m\)-active path \(\pi^m\) from \(\X^m\) to \(\Y^m\). Let \(V^m\) be a cluster of size greater than 4.  By Proposition \ref{prop:reduce_complexity_path}, there exists a compatible graph \(\Gm'\) that contains a \(\Z^m\)-active path \({\pi^m}'\) from \(\X^m\) to \(\Y^m\), where \({\pi^m}'\) does not contain two colliders nor two forks nor two chains on $V^m$.  By Corollary \ref{cor:move_fork_up}, we may assume that if a fork is present, it is located at \(V_1\).  Next, consider the set  
        \[\C \coloneqq \left\{ v \mid \exists ~ C_c \text{ collider on } {\pi^m}' \text{ such that } V_v \in \Desc(C_c, \Gm') \cap \Anc(\\Z^m, \Gm') \right\}.\]  
        For every \(v \in \C\), we apply Property \ref{prop:move_arrow_down} to move all incoming arrows into \(V_v\) to \(V_{\max \C}\). This transformation yields a compatible graph in which any collider on \({\pi^m}'\) (if present) is now located at \(V_{\max \C}\). Since this collider remains an ancestor of \(\Z^m\), it remains non-blocking. Similarly, any other colliders (if present) on \({\pi^m}'\) remain ancestors of \(\Z^m\) and thus remain non-blocking.  Now, let \(V_i\) be the element of \(V^m\) where \({\pi^m}'\) contains a chain (if any). For all \(V_v \notin \{V_1, V_{\max \C}, V_i\}\), we use Propositions \ref{prop:move_arrow_up} and \ref{prop:move_arrow_down} to redirect all arrows towards \(V_v\) to \(V_{\#V^m}\). As a result, we obtain a compatible graph in which at most 4 elements of \(V^m\) have arrows. Consequently, we can remove all other elements, yielding a graph compatible with \(\Gc\) in which the cardinality of \(V^m\) has been reduced to 4.  
        Repeating this process for all other clusters \(V^m\), we obtain a compatible graph with \(\G^c_{\leq4}\) that still contains a \(\Z^m\)-active path from \(\X^m\) to \(\Y^m\).
    \end{itemize}
\end{proof}

\newpage
\subsection{A graphical criterion for d-separation in cluster-DAGs}

\begin{definition}[Minimal Compatible Graph]
\label{def:gm_min}
    Let $\Gc$ be an admissible cluster-DAG. Its corresponding minimal compatible graph is $\Gt = \left( \mathcal{V}_{\min}, \mathcal{E}_{\min}  \right)$, the mixed graph defined by the following procedure:
    \begin{itemize}
        \item $\mathcal{V}_{\min} = \V^m \coloneqq  \bigcup_{V^C \in \mathcal{V}^C} \left\{ V_1, \dots, V_{\#V^C}\right\}$. 
        \begin{enumerate}
            \item For all dashed-bidirected-arrows $U^C \dashleftrightarrow V^C$ in $\Gc$, add the dashed-bidirected-arrow $U_i \dashleftrightarrow V_j$ for all $i,j \in \{1,\dots, \#U^C\} \times \{1,\dots, \#V^C\}$ such that $U_i \neq V_j$.\label{def:gm_min:1}
            
            \item For all self-loop $\leftselfloop V^C$, add the arrow $V_i \rightarrow V_j$ for all $i,j \in \{1,\dots, \#V^C\}^2$ such that $i <j$. \label{def:gm_min:2}
            
            \item For all arrows $U^C \rightarrow V^C$, with $U^C \neq V^C$, add the arrow $U_1 \rightarrow V_{\#V^C}$. \label{def:gm_min:3}
            
        \end{enumerate}
    \end{itemize}
\end{definition}

Figure \ref{fig:example_gm_min} gives an example of a cluster DAG and its minimal compatible graph.

\begin{figure}[t]
    \begin{tikzpicture}[ ->, >=stealth, scale=3]
        % Nodes
        \node (A) at (0, 0) {${}^3A$};
        \node (B) at (1, 0) {${}^2B$};
        \node (C) at (0.5, -0.7071) {${}^1C$};

        % Edges
        \draw (C) -- (A);
        \path (C) edge [<->, dashed, bend right = 15] node {} (B);
        \path (A) edge[bend left=35] (B);
        \path (B) edge[bend left=35] (A);
        \path (A) edge [->, loop left, looseness=5, in=155, out=240] node {} (A);

        \end{tikzpicture} \hfill
        \begin{tikzpicture}[->, >=stealth, scale=1.1]
        % Micro-level Nodes
        \node[anchor=center] (A1) at (0.0000, 0.6667) {$A_1$};
        \node[anchor=center] (A2) at (-.5774, -.3333) {$A_2$};
        \node[anchor=center] (A3) at (0.5774, -.3333) {$A_3$};
        
        \node[anchor=center] (B1) at (3.0000, 0.6667) {$B_1$};
        \node[anchor=center] (B2) at (3.0000, -.3333) {$B_2$};
        
        \node[anchor=center] (C)  at (1.50000, -2.0284) {$C_1$};
        
        \node[draw, dashed, ellipse, scale= 7.0, xscale=1, yscale=1, anchor=center]  at (0, 0) {};
        \node[draw, dashed, ellipse, scale= 1.85, xscale=2, yscale=4, anchor=center]  at (3, 0.1) {};
        \node[draw, dashed, ellipse, scale= 3.2, xscale=1, yscale=1, anchor=center]  at (1.5000, -2.0284) {};

        \path (C) edge (A3);
        \path (C) edge [<->, dashed, bend right = 75] node {} (B1);

        \path (C) edge [<->, dashed, bend right = 20] node {} (B2);

        \draw (A1) -- (B2);
        \draw (A1) -- (A2);
        \draw (A1) -- (A3);
        \draw (A2) -- (A3);
        \draw (B1) -- (A3);

    \end{tikzpicture}
        \caption{Example of Cluster-DAG (left) its minimal compatible graph (right).}
        \label{fig:example_gm_min}
\end{figure}
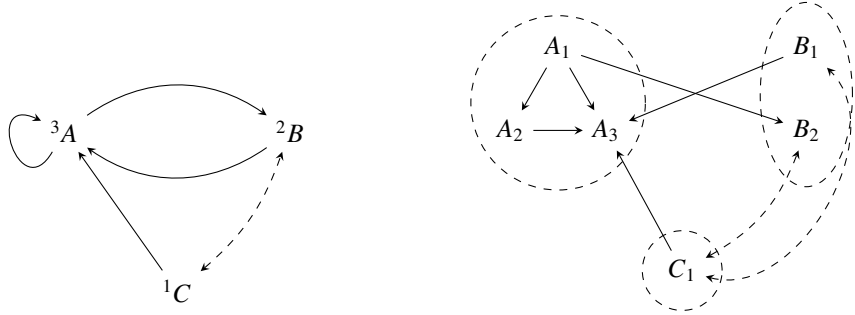

\begin{lemma}
\label{lemma:G3_compatible}
    Let $\Gc$ be an admissible cluster-DAG. Then, $\Gt$, its minimal compatible graph, is compatible with $\Gc$.
\end{lemma}

\begin{proof}
By construction, all arrows in $\Gc$ are represented in $\Gt$. Moreover, all arrows that are added correspond to an arrow in $\Gc$. Therefore, we only need to check for acyclicity to prove that $\Gt$ is a compatible graph with $\Gc$. 

By contradiction, let us assume that $\Gt$ contains a cycle $\pi$. First, we know that $\pi$ is not within a single cluster. Indeed, in each cluster $V$, $\Gt_{\mid V}$ is $V_{\#V}$-rooted tree. Thus, $\pi$ encounters at least two clusters and has an arrow between two clusters. Let $A_1 \rightarrow B_{\#B}$ be such an arrow. Necessarily, $\#A^C = 1$. Indeed, if $\#A^C \geq 2$, then no arrow in $\Gt$ is pointing on $A_1$. Similarly, we can show that $\#B^C = 1$. Thus, the next arrow cannot be pointing into $B^C$, thus the next arrow is also an arrow between two clusters. By induction, we show that all the clusters encountered by $\pi$ have a cardinal of $1$. This contradicts the hypothesis that $\Gc$ is an admissible cluster-DAG. Therefore, $\Gt$ is acyclic. 

Therefore, $\Gt$ is a compatible graph with $\Gc$.
\end{proof}

\begin{lemma}
\label{lemma:gm_cup_g3_compatible}
    Let $\Gc$ be a cluster-DAG, $\Gt$ be its corresponding minimal compatible graph, and $\Gm$ be a compatible graph. Then,  $\Gt \cup \Gm$  is a compatible graph.
\end{lemma}

\begin{proof}
Since $\Gm$ and $\Gt$ are compatible, then we only need to check for acyclicity. Let label indices according to Notation \ref{notation:base}. Let $a$ be an arrow in $\Gt$ that is not in $\Gm$. We distinguish three cases:
\begin{itemize}
    \item If $a$ is a dashed-bidirected arrow, then adding $a$ does not create a cycle.
    
    \item If $a$ is a directed arrow inside a cluster, since $\Gm$ follows Notation \ref{notation:base}, then $a$ can be added without creating a cluster.
    
    \item Otherwise, $a$ corresponds to an arrow between two clusters $V^m$ and $U^m$. $\Gm$ is compatible thus it also contains an arrow from $U^m$ to $V^m$. By applying Propositions \ref{prop:move_arrow_up} and \ref{prop:move_arrow_down}, we see that we can add $a$ without creating a cycle.
\end{itemize}

Therefore, $\Gt \cup \Gm$  is a compatible graph.
\end{proof}

\begin{definition}
\label{def:unfolded_graph}
    Let $\Gc$ be an admissible cluster-DAG and $\Gt =  (\mathcal{V}_{\min}, \mathcal{E}_{\min})$ be its corresponding minimal compatible graph. Its corresponding unfolded graph is $\Gcu = \left( \mathcal{V}_{\text{u}}, \mathcal{E}_{\text{u}}  \right)$, the mixed graph defined by the following procedure:
    \begin{itemize}
        \item $\mathcal{V}_{\text{u}} \coloneqq  \mathcal{V}_{\min}$.
        
        \item Let us consider the following set:
        \[
            \E_{\text{to choose}} \coloneqq \left\{
            U_i \rightarrow V_j \,\middle|\,
            \begin{cases}
                U^C \rightarrow V^C \subseteq \Gc, \text{ and,} \\
                U_i \rightarrow V_j \text{ does not create a cycle in } \Gt
            \end{cases}
            \right\}
        \]

        \noindent Then $\E_u \coloneqq \E_{\min} \cup \E_{\text{ "to choose"}}$
    \end{itemize}
\end{definition}

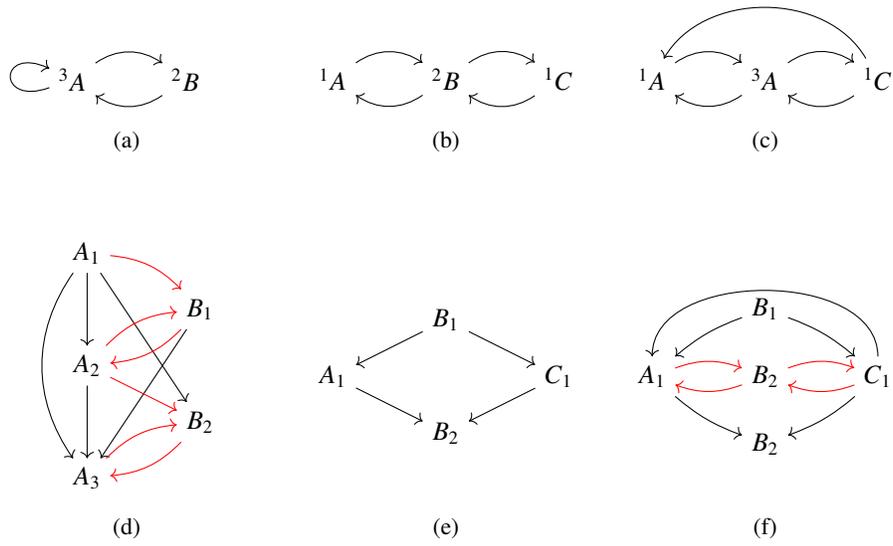
\begin{figure}
    % En haut à gauche
    \begin{subfigure}{0.3\textwidth}
        \centering
        \begin{tikzpicture}[->, scale = 1.5]
            \node (A) at (0,2) {${}^3A$};
            \node (B) at (1,2) {${}^2B$};
            \path (A) edge [bend left = 35] (B);
            \path (B) edge [bend left = 35] (A);
            \path (A) edge [->, loop left, looseness=5, in=155, out=200, min distance=0.5cm] node {} (A);
            \path (B) edge [->, loop right, looseness=5, in=20, out=65, min distance=0.5cm, draw=none] node {} (B); 
        \end{tikzpicture}
        \caption{}
        \label{fig:example_gcu:1}
    \end{subfigure}
    \hfill
    % En haut au milieu
    \begin{subfigure}{0.3\textwidth}
        \centering
        \begin{tikzpicture}[->, scale = 1.5]
            \node (A) at (0,2) {${}^1A$};
            \node (B) at (1,2) {${}^2B$};
            \node (C) at (2,2) {${}^1C$};
            \path (A) edge [bend left = 35] (B);
            \path (B) edge [bend left = 35] (A);
            \path (B) edge [bend left = 35] (C);
            \path (C) edge [bend left = 35] (B);
        \end{tikzpicture}
        \caption{}
        \label{fig:example_gcu:2}
    \end{subfigure}
    \hfill
    % En haut à droite
    \begin{subfigure}{0.3\textwidth}
        \centering
        \begin{tikzpicture}[->, scale = 1.5]
            \node (A) at (0,2) {${}^1A$};
            \node (B) at (1,2) {${}^3A$};
            \node (C) at (2,2) {${}^1C$};
            \path (A) edge [bend left = 35] (B);
            \path (B) edge [bend left = 35] (A);
            \path (B) edge [bend left = 35] (C);
            \path (C) edge [bend left = 35] (B);
            \path (C) edge [bend right = 60] (A);
        \end{tikzpicture}
        \caption{}
        \label{fig:example_gcu:3}
    \end{subfigure}

    \vspace{1cm}

    % En bas à gauche
    \begin{subfigure}{0.3\textwidth}
        \centering
        \begin{tikzpicture}[->, scale = 1.5]
            \node (A1) at (0,1) {$A_1$};
            \node (A2) at (0,0) {$A_2$};
            \node (A3) at (0,-1) {$A_3$};
            \node (B1) at (1,.5) {$B_1$};
            \node (B2) at (1,-.5) {$B_2$};
            \path (A1) edge  (A2);
            \path (A2) edge  (A3);
            \path (A1) edge [bend right = 35]  (A3);
            \path (A1) edge  (B2);
            \path (B1) edge  (A3);
            
            \path (A1) edge [bend left = 20, color=red] (B1);
            \path (B1) edge [bend left = 20, color=red] (A2);
            \path (A2) edge [bend left = 20, color=red] (B1);
            \path (A2) edge [color=red] (B2);
            \path (B2) edge [bend left = 20, color=red] (A3);
            \path (A3) edge [bend left = 20, color=red] (B2);
        \end{tikzpicture}
        \caption{}
        \label{fig:example_gcu:4}
    \end{subfigure}
    \hfill
    % En bas au milieu
    \begin{subfigure}{0.3\textwidth}
        \centering
        \begin{tikzpicture}[->, scale = 1.5]
            \node (A1) at (0,0) {$A_1$};
            \node (B1) at (1,.5) {$B_1$};
            \node (B2) at (1,-.5) {$B_2$};
            \node (C1) at (2,0) {$C_1$};
            \node (blank1) at (0,1)  {}; 
            \node (blank2) at (0,-1)  {};
            
            \path (A1) edge  (B2);
            \path (B1) edge  (A1);
            \path (B1) edge  (C1);
            \path (C1) edge  (B2);

        \end{tikzpicture}
        \caption{}
        \label{fig:example_gcu:5}
    \end{subfigure}
    \hfill
    % En bas à droite
    \begin{subfigure}{0.3\textwidth}
        \centering
        \begin{tikzpicture}[->, scale = 1.5]
            \node (A1) at (0,0) {$A_1$};
            \node (B1) at (1,.6) {$B_1$};
            \node (B2) at (1,0) {$B_2$};
            \node (B3) at (1,-.6) {$B_2$};
            \node (C1) at (2,0) {$C_1$};
            \node (blank1) at (0,1)  {}; 
            \node (blank2) at (0,-1)  {};
            
            \path (A1) edge [bend right = 10] (B3);
            \path (B1) edge [bend right = 10] (A1);
            \path (B1) edge [bend left = 10] (C1);
            \path (C1) edge [bend left = 10] (B3);
            \path (C1) edge [bend right =90] (A1);
            
            \path (A1) edge [bend left = 20, color=red] (B2);
            \path (B2) edge [bend left = 20, color=red] (A1);
            
            \path (B2) edge [bend left = 20, color=red] (C1);
            \path (C1) edge [bend left = 20, color=red] (B2);

        \end{tikzpicture}
        \caption{}
        \label{fig:example_gcu:6}
    \end{subfigure}
    \caption{On the first row (Figures \ref{fig:example_gcu:1}, \ref{fig:example_gcu:2} and \ref{fig:example_gcu:3}), three examples of C-DAG are given. On the second row (respectively, Figures \ref{fig:example_gcu:4}, \ref{fig:example_gcu:5} and \ref{fig:example_gcu:6}), we represents the corresponding unfolded graphs. The black arrows corresponds to $\Gt$, whereas the \textcolor{red}{red arrows} represent the "to choose" arrows. Lemma \ref{lemma:gm_subgraph_gcu} and Figure \ref{fig:example_gcu:5} show that there is no graph compatible with the C-DAG depicted in Figure \ref{fig:example_gcu:2} such that $A_1$ and $C_1$ are connected by a directed path. Similarly, Lemma \ref{lemma:gm_cup_g3_compatible} and Figure \ref{fig:example_gcu:6} show that there is no graph $\Gm$ compatible with the C-DAG depicted in Figure \ref{fig:example_gcu:3} such that $A_1 \in \Anc(C_1, \Gm)$. }
\end{figure}

\begin{remark}
    In general, the complexity of building $\Gcu$ is polynomial with respect to the size of $\Gc$ and the size of each cluster. It can be intractable if some clusters are very large. Nonetheless, in practice, we will work with $\G^C_{\leq 4}$. In this setting, $\Gcu$ is computed in polynomial time with respect to the size of $\Gc$. 
\end{remark}

\begin{lemma}
\label{lemma:gm_subgraph_gcu}
    Let $\Gc$ be an admissible cluster-DAG and $\Gcu$ be its corresponding unfolded graph. Then, any compatible graph $\Gm$ is a subgraph of $\Gcu$ up to a permutation of indices in each cluster.
\end{lemma}

\begin{proof}
    Let $\Gm =\left( \mathcal{V}^m, \mathcal{E}^m  \right)$ be a compatible graph. By definition, we already know that $\mathcal{V}^m = \mathcal{V}_{\text{u}}$.  $\Gm$ is a DAG, thus, in each cluster, we can permute the indices of the vertices so that the topological order of $\Gm$ agrees with the order of the indices. Let $a$ be an arrow of $\Gm$. We distinguish three cases:
    \begin{itemize}
        \item If $a$ is a dashed-bidirected-arrow, then $a$ is also in $\Gcu$ because $\Gcu$ contains all possible dashed-bidirected-arrows.
        
        \item If $a$ corresponds to a self-loop $\leftselfloop V^C$ in $\Gc$. Necessarily, in $\Gm$, $a = V_i \rightarrow V_j$ with $i<j$. Thus, $a$ corresponds to an arrow added during step \ref{def:gm_min:2} of Definition \ref{def:gm_min}. Therefore, $a$ is also an arrow in $\Gcu$.
        
        \item Otherwise, $a$ corresponds to an arrow $U^C \rightarrow V^C$, with $U^C \neq V^C$. We distinguish two cases:
        \begin{itemize}
            \item If $a$ is added at step \ref{def:gm_min:3} of definition \ref{def:gm_min}, then $a$ is also an arrow in $\Gcu$.
            
            \item Otherwise, by Lemma \ref{lemma:gm_cup_g3_compatible}, we know that $\Gt \cup \Gm$ is compatible. Thus  $\Gt \cup \Gm$ is acyclic. Since $\Gt \cup a$ is a subgraph of $\Gt \cup \Gm$, we can conclude that $a$ does not create a cycle in $\Gt$. Thus, $a \in \E_{\text{"to choose"}}$. Therefore $a$ is  an arrow in $\Gcu$
        \end{itemize}
    \end{itemize}
    Hence, $\mathcal{E}^m \subseteq \mathcal{E}_{\text{u}}$. Therefore, $\Gm$ is a subgraph of $G_{\text{u}}$.
    
\end{proof}

Remark that 
    $\Gcu$ is not acyclic, except if $\Gc$ is acyclic. If $\Gc$ is acyclic, then $\Gcu$ is the compatible graph with as many arrows as possible. Thus, in this case, we recover the results of \cite{anand_causal_2023}.

\begin{remark}
    If $\mathcal{X}^C \notind_{\Gc} \mathcal{Y}^C \mid \mathcal{Z}^C$, then there exists a compatible graph that contains a connecting path. Since $\Gcu$ is a supergraph of this graph, the path also exists in $\Gcu$. Thus, d-separation in $\Gcu$ is sound.
\end{remark}

\begin{theorem}[Graphical Criterion]
\label{th:graphical_criterion}
    Let $\Gc = \left( \mathcal{V}^C, \mathcal{E}^C  \right)$ be an admissible cluster-DAG. Let $\Gt$ be its corresponding minimal compatible graph. Let $\Gcu$ be the corresponding unfolded graph. Let $\X^C,\Y^C$ and $\Z^C$ be pairwise distinct subsets of nodes of $\Gc$. Then the following properties are equivalent:
    
    \begin{enumerate}
        \item $\mathcal{X}^C \notind_{\Gc} \mathcal{Y}^C \mid \mathcal{Z}^C$, \label{th:graphical_criterion:1}
        
        \item $\Gcu$ contains a structure of interest $\sigma^m$ that connects $\X^m$ to $\Y^m$ under $\Z^m$ such that $\Gt \cup \sigma^m$ is acyclic. \label{th:graphical_criterion:2}
    \end{enumerate}
\end{theorem}

\begin{proof}
Let us prove the two implications:
\begin{itemize}
    \item $\ref{th:graphical_criterion:1} \Rightarrow \ref{th:graphical_criterion:2}$: If $\mathcal{X}^C \notind_{\Gc} \mathcal{Y}^C \mid \mathcal{Z}^C$, then there exists a compatible graph $\Gm$ in which there exists a structure of interest $\sigma^m$ which connects $\X^m$ and $\Y^m$ under $\Z^m$. By Lemma \ref{lemma:gm_subgraph_gcu}, $\Gm$ is a subgraph of $\Gcu$. Thus $\sigma^m$ is also a structure of interest which connects $\X^m$ and $\Y^m$ under $\Z^m$ in $\Gcu$.
    
    By Lemma \ref{lemma:gm_cup_g3_compatible}, $\Gt \cup \Gm $ is a compatible graph. Moreover, $\Gt \cup \sigma^m \subseteq \Gt \cup \Gm $. Thus, by Proposition \ref{prop:remove_arrow}, $\Gt \cup \sigma^m$ is a compatible graph. Therefore,  $\Gt \cup \sigma^m$ is acyclic.

    \item $\ref{th:graphical_criterion:2} \Rightarrow \ref{th:graphical_criterion:1}$: $\Gt \cup \sigma^m$ is acyclic. By Lemma \ref{lemma:G3_compatible}, $\Gt$ is compatible. Moreover, all arrows from $\sigma^m$ come from $\Gcu$. Therefore,  $\Gt \cup \sigma^m$ is a compatible graph. By Theorem \ref{th:new_d_sep}, $\mathcal{X}^C \notind_{\Gt \cup \sigma^m} \mathcal{Y}^C \mid \mathcal{Z}^C$. Therefore, $\mathcal{X}^C \notind_{\Gc} \mathcal{Y}^C \mid \mathcal{Z}^C$. 
    
\end{itemize}
\end{proof}

\section{Solving Problem \ref{pb:main_goal}}

\begin{theorem}
\label{th:final_theorem}
    Let $\Gc = \left( \mathcal{V}^C, \mathcal{E}^C  \right)$ be an admissible cluster-DAG.  Let $\Gt$ be its corresponding minimal compatible graph. Let $\Gcu$ be the corresponding unfolded graph. Let $\X^C,\Y^C$ and $\Z^C$ be pairwise distinct subsets of nodes of $\Gc$. Let $\mathcal{A}$ and $\mathcal{B}$ be subsets of nodes of $\Gc$. Then the following properties are equivalent:
    
    \begin{enumerate}
        \item $\X^C \notind_{\Gc_{\overline{\mathcal{A}^C} \underline{\mathcal{B}^C}}} \Y^C \mid \Z^C$ \label{th:final_theorem:1}
        
        \item $\Gcu_{\overline{\mathcal{A}^m} \underline{\mathcal{B}^m}}$ contains a structure of interest $\sigma^m$ that connects $\X^m$ to $\Y^m$ under $\Z^m$ such that $\Gt \cup \sigma^m$ is acyclic. \label{th:final_theorem:2}
    \end{enumerate}
\end{theorem}

\begin{proof}
Let us prove the two implications:

\begin{itemize}
    \item $\ref{th:final_theorem:1} \Rightarrow \ref{th:final_theorem:2}$: If $\X^C \notind_{\Gc_{\overline{\mathcal{A}^C} \underline{\mathcal{B}^C}}} \Y^C \mid \Z^C$, then there exists a compatible graph $\Gm$ such that $\Gm_{\overline{\mathcal{A}^m} \underline{\mathcal{B}^m}}$ contains a structure of interest $\sigma^m$ which connects $\X^m$ and $\Y^m$ under $\Z^m$. By Lemma \ref{lemma:gm_subgraph_gcu} $\sigma^m \subseteq \Gm_{\overline{\mathcal{A}^m} \underline{\mathcal{B}^m}} \subseteq \Gm \subseteq \Gcu$. Moreover, since $\sigma^m \subseteq \Gm_{\overline{\mathcal{A}^m} \underline{\mathcal{B}^m}}$, we know that $\sigma^m$ does not contain any incoming arrow in $\A^m$ and no outgoing arrow from $\B^m$. Therefore, $\sigma^m \subseteq \Gcu_{\overline{\mathcal{A}^m} \underline{\mathcal{B}^m}}$ and $\sigma^m$ connects $\X^m$ and $\Y^m$ under $\Z^m$.
    
    By Lemma \ref{lemma:gm_cup_g3_compatible}, $\Gt \cup \Gm $ is a compatible graph, thus acyclic. Moreover, $\Gt \cup \sigma^m \subseteq \Gt \cup \Gm $. Therefore,  $\Gt \cup \sigma^m$ is acyclic.
    
    \item $\ref{th:final_theorem:2} \Rightarrow \ref{th:final_theorem:1}$: $\Gt \cup \sigma^m$ is acyclic. By Lemma \ref{lemma:G3_compatible}, $\Gt$ is compatible. Moreover, all arrows from $\sigma^m$ come from $\Gcu$. Therefore,  $\Gt \cup \sigma^m$ is a compatible graph. 
    
    Moreover, since $\sigma^m \subseteq \Gcu_{\overline{\mathcal{A}^m} \underline{\mathcal{B}^m}}$, we know that $\sigma^m$ does not contain any incoming arrow in $\A^m$ and no outgoing arrow from $\B^m$. Thus, $ \sigma^m \subseteq (\Gt \cup \sigma^m)_{\overline{\mathcal{A}^m} \underline{\mathcal{B}^m}}$. Since $\sigma^m$ connects $\X^m$ and $\Y^m$ under $\Z^m$, we can conclude that $\X^m \notind_{(\Gt \cup \sigma^m)_{\overline{\mathcal{A}^m} \underline{\mathcal{B}^m}}} \Y^C \mid \Z^m$. 
    
    Therefore, $\X^C \notind_{\Gc_{\overline{\mathcal{A}^C} \underline{\mathcal{B}^C}}} \Y^C \mid \Z^C$.
\end{itemize}
\end{proof}

\begin{remark}
    Since all do-calculus rules can be written with d-separation under mutilations. Theorem \ref{th:final_theorem} gives a complete characterization of common do-calculus in cluster-DAGs.
\end{remark}

\section{Conclusion}
In this note, we provide a graphical criterion for identifying the total effect in cluster DAGs where cycles are allowed (i.e., the partition into clusters is not admissible). To the best of our knowledge, this is the first sound and complete criterion in this setting.

\bibliography{References}

\end{document}